 \documentclass[final,5p,times,twocolumn,authoryear]{elsarticle}

\usepackage{amssymb}

\usepackage{amsmath,amssymb,amsfonts}
\usepackage[justification=centering]{caption}
\usepackage{subfigure}
\usepackage{verbatim}
\usepackage{hyperref}
\usepackage{wrapfig}
\hypersetup{
	colorlinks=true,
	linkcolor=blue,
	filecolor=blue,      
	urlcolor=blue,
	citecolor=cyan,
}

\usepackage{makecell}
\usepackage{multirow}

\newdefinition{assumption}{ {Assumption}}
\newtheorem{theorem}{ {Theorem}}
\newtheorem{lemma}{ {Lemma}}
\newdefinition{definition}{ {Definition}}
\newdefinition{remark}{ {Remark}}
\newdefinition{example}{ {Example}}
\newproof{proof}{Proof}
\newcommand{\sign}[1]{\mathrm{sgn}(#1)}

\journal{Arxiv }
\begin{document}

\begin{frontmatter}



\title{Distributed online optimization for heterogeneous linear multi-agent systems with coupled constraints \tnoteref{t1}}


\author[1,2]{Yang Yu}
\ead{1910639@tongji.edu.cn}
\author[1,2,3]{Xiuxian Li\corref{cor1}}
\ead{xli@tongji.edu.cn}
\author[1,2,3]{Li Li}
\ead{lili@tongji.edu.cn}
\author[4]{Lihua Xie
	}
\ead{elhxie@ntu.edu.sg}

\cortext[cor1]{Corresponding author.}
\address[1]{Department of Control Science and Engineering, College of Electronics and Information Engineering, Tongji University, Shanghai 201804, China}
\address[2]{Shanghai Research Institute for Intelligent Autonomous Systems, Shanghai, 201210, China}
\address[3]{Shanghai Institute of Intelligent Science and Technology, Tongji University, Shanghai 201804, China}
\address[4]{School of Electrical and Electronic Engineering, Nanyang Technological University, Singapore 639798}


\begin{abstract}
This paper studies a class of distributed online convex optimization problems for heterogeneous linear multi-agent systems. Agents in a network, knowing only their own outputs, need to minimize the time-varying costs through neighboring interaction subject to time-varying coupled inequality constraints. Based on the saddle-point technique, we design a  continuous-time distributed controller which is shown to achieve constant regret bound and sublinear fit bound, matching those of the standard centralized online method. We further extend the control law to the event-triggered communication mechanism and show that the constant regret bound and sublinear fit bound are still achieved while reducing the communication frequency. Additionally, we study the situation of communication noise, i.e., the agent's measurement of the relative states of its neighbors is disturbed by a noise. It is shown that, if the noise is not excessive, the regret and fit bounds are unaffected, which indicates the  controller's noise-tolerance capability to some extent. Finally,  a numerical simulation is provided to support the theoretical conclusions.
\end{abstract}



\begin{keyword}
 Online convex optimization \sep multi-agent systems \sep event-triggered communication \sep time-varying constraints \sep noisy measurement


\end{keyword}

\end{frontmatter}


\section{Introduction}\label{int}
	With the increase of data scale and the development of network technology, the problems of accomplishing global tasks collaboratively through multiple nodes has attracted more and more attention, and distributed optimization is one of such problems \citep{nedic2018,yang2019}. Compared with optimization via a single-node, distributed optimization enables each computing node to enjoy the computing power of the entire network at the cost of additional network communication. In addition, distributed optimization can effectively avoid single-point failures in centralized computation and improve performance bottlenecks \citep{chen2012}. Distributed optimization is widely used in power systems \citep{anoh2020}, industrial smart manufacturing \citep{mao2022}, robot formation \citep{sun2022}, etc.
	
	\begin{table*}
		\small
		\caption{Comparison of the algorithm proposed in this paper to related works measured by static regret and fit}
		\label{tab_1}
			\begin{tabular*}{\linewidth}{@{}c|c|c|c|l@{}}
				\hline
				Reference & Controller type & Problem type & Constraint type & Regret and fit\\
				\hline
				\citet{pater2017} & Continuous & Centralized & Feasible set constraints & $\mathcal{R}^T=\mathcal{O}(1)$\\
				\hline
				\citet{pater2020} & Continuous & Distributed & Uncoupled inequality constraints &\makecell[l]{$\mathcal{R}^T=\mathcal{O}(\sqrt{T})$,$\mathcal{F}^T=\mathcal{O}(\sqrt{T})$} \\
				\hline
				\citet{yi2020} & Discrete & Distributed & Coupled inequality constraints & \makecell[l]{$\mathcal{R}^T=\mathcal{O}(T^{\max\{\kappa,1-\kappa\}})$, $\mathcal{F}^T = \mathcal{O}(T^{\max\{\kappa,1-\kappa\}})$\\ for $\kappa\in(0,1)$}\\
				\hline
				\citet{li2022} & Discrete & Distributed & Coupled inequality constraints & \makecell[l]{$\mathcal{R}^T= \mathcal{O}(T^{\max\{\kappa,1-\kappa\}})$, $\mathcal{F}^T = \mathcal{O}(T^{\max\{\frac{1}{2}+\frac{\kappa}{2},1-\frac{\kappa}{2}\}})$\\ for $\kappa\in(0,1)$} \\
				\hline
				This paper & Continuous &Distributed & Coupled inequality constraints &\makecell[l]{ $\mathcal{R}^T=\mathcal{O}(1)$,$\mathcal{F}^T=\mathcal{O}(\sqrt{T})$}\\
				\hline
			\end{tabular*}
	\end{table*}
	
	Due to the complexity and uncertainty of the environment, optimization goals and constraints are usually time-varying, which make it difficult to formulate a comprehensive theoretical model and use classical algorithmic theory to optimize. The online optimization considered in this paper optimizes by making decisions while learning, and dynamically adjusting the strategy \citep{ss2011}. 
	In the discrete-time setting, distributed online optimization problem with global feasible set constraints was studied in \citet{yan2013} and the proposed projection-based online subgradient descent algorithm achieves $\mathcal{O}(\sqrt{T})$ static regret bound, which matches the bound of the centralized algorithm \citep{zink2003}. Further, \citet{kop2015} studied the problem with local feasible set constraints and obtained the same bound.
	For the distributed online optimization problem with local inequality constraints, a consensus-based primal-dual subgradient algorithm was proposed in \citet{yuan2018} which achieves sublinear regret bound. In order to address the time-varying coupling constraints, \citet{yi2020} developed a unique decentralized dynamic online  mirror descent algorithm to achieve sublinear dynamic regret. In the continuous-time setting, an online projected gradient controller was designed in \citet{pater2017}, which achieves the constant regret bound independent of $T$. And this algorithm was extended in \citet{pater2020} to the case of multi-agent systems with $\mathcal{O}(\sqrt{T})$ regret bound. Moreover, distributed online optimization problems have also been studied in the face of various complexities such as communication noise \citep{cao2021}, signal transmission delay \citep{cao2022},  unavailable gradients \citep{yi20212}, etc., and more details can be referenced in a recent survey \citep{li2022sur}.

	Most of the above works concentrate on online optimization in the absence of system dynamics, which can be viewed as information layer problems. However, agents tend to have their complex physical dynamics, such as Euler–Lagrangian (EL) dynamics for mobile robots and quadrotors \citep{ ray2021}. So the combination of online optimization and dynamics should be analyzed and controlled properly as a physical layer problem. 
	For high-order multi-agent systems, a distributed Proportional-Integral (PI) controller was designed in \citet{deng2016}  to deal with the distributed online optimization problem. There are also some centralized algorithms that consider double integrators \citep{ll2021} and linear systems \citep{col2020,non2021}.
	The distributed setting for online optimization algorithm with linear dynamics remains an open problem.
	The following are the contributions of this paper.
	\begin{itemize}
		\item This study investigates the distributed online optimization of heterogeneous multi-agent systems with time-varying coupled inequality constraints for the first time, in contrast to the centralized online optimization algorithms for linear systems \citep{col2020,non2021}. 
		Agents can reach constant regret bound and $\mathcal{O}(\sqrt{T})$ fit bound by  the designed continuous controllers. In comparison, majority of existing distributed online optimization algorithms \citep{yi2020,li2022} with coupled inequality constraints only achieve inferior sublinear regret bounds.
		
		\item This study provides an event-triggered approach based on the continuous-time online optimization controller  to reduce the bandwidth consumption. The constant regret bound and $\mathcal{O}(\sqrt{T})$ fit bound are still reached in the case of event-triggered communication.
		\item Furthermore, the impact of communication noise on the systems is studied. In this case, the constant regret bound and $\mathcal{O}(\sqrt{T})$ fit bound are still guaranteed if the noise is not excessive, demonstrating that the  controller has a certain tolerance for noise.
	\end{itemize}
	
	A preliminary version of the paper was presented at a conference \citep{yy2022b}. The current paper extends it in the following aspects. First, this paper provides detailed proofs for the main results which are not provided in \citet{yy2022b}.  Second, the algorithm is further extended here to handle the identical output case. Third, this paper also studies the scenario with communication noise, which is not addressed in \citet{yy2022b}. Finally, compared with \citet{yy2022b}, 
	 a improved numerical simulation is presented here.
	
	The rest of the paper is structured as follows. Section \ref{pre} contains preliminaries. Following that, we describe the heterogeneous linear systems under study, define the online convex optimization problem and give some supportive lemmas in Section \ref{pro}. In Section \ref{mai}, the continuous and event-triggered control laws are developed and the influence of noise signals is studied.  A numerical example is provided in Section \ref{sim} and the conclusion is covered in Section \ref{res}.

\section{Preliminaries}\label{pre}

\subsection{Notations}
	$\mathbb{R}$, $\mathbb{N}_+$, $\mathbb{R}^n$, $\mathbb{R}_+^n$, $\mathbb{R}^{m\times n}$ stand for the set of real numbers, positive integers, $n$-dimensional real vectors, $n$-dimensional non-negative real vectors, and $m\times n$-dimensional real matrices, respectively. $\boldsymbol0_n$ (resp. $\boldsymbol1_n$) denotes the all-zero (resp. all-one) column  vector  of dimension $n$ and $I_n$ denotes the $n$-dimensional identity matrix. The Euclidean norm (resp. 1-norm) of vector $x$ is denoted by $\|x\|$ (resp. $\|x\|_1$). $[n]$ denotes the set $\{1,\dots,n\}$ for any $n\in\mathbb{N}_+$. $col(x_1,\dots,x_n)$ represents a column vector by stacking vectors $x_i$, $i\in[n]$. $diag(A_1,\dots,A_n)$ denotes a block diagonal matrix with diagonal blocks of $A_1$, $\dots$, $A_n$. The Kronecker product of matrices $A$ and $B$ is denoted by $A\otimes B$.
	$P_\mathcal{S}(x)$ denotes the Euclidean projection of a vector $x\in \mathbb{R}^n$ onto the set $\mathcal{S} \subseteq \mathbb{R}^n$, i.e., 
	$P_\mathcal{S}(x)=argmin_{v\in\mathcal{S}}\|x-v\|^2$.
	For simplicity, denote $P_{\mathbb{R}_+^n}(\cdot)$ as $[\cdot]_+$.	
	 Define the set-valued sign function $\mathrm{sgn}(\cdot)$ as follows:
	\begin{align*}
		\sign{x}:=\partial\|x\|_1=
		\begin{cases}
			-1,&\mathrm{if}\  x<0,\\
			[-1,1],&\mathrm{if}\  x=0,\\
			1,&\mathrm{if}\  x>0.
		\end{cases}
	\end{align*}

\subsection{Graph Theory}
	An undirected graph $\mathcal{G} = (\mathcal{V}, \mathcal{E}, \mathcal{A})$ is used to describe the communication network of a multi-agent system with $N$ agents, where $\mathcal{V} = \{v_1, \dots, v_N\}$ is a node set. $\mathcal{E} \in \mathcal{V} \times \mathcal{V}$ denotes the edge set. If the information can be exchanged between $v_j$ and $v_i$, then $(v_j,v_i) \in \mathcal{E}$ with $a_{ij}=1$ denoting its weight, otherwise  $a_{ij}=0$. $\mathcal{A}= [a_{ij}]\in \mathbb{R}^{N \times N}$ is the adjacency matrix. $\mathcal{G}$ is connected if there exists a path from any node to any other node in $\mathcal{V}$.
	
	\begin{assumption} \label{asp1}
		The communication graph $\mathcal{G}$ is undirected and connected.	
	\end{assumption}

\section{Problem Formulation}\label{pro}

\subsection{System Dynamics}

	Consider a multi-agent system consisting of $N$ heterogeneous agents. The dynamic of each agent $i\in\mathcal{V}$ is modeled by the following linear system:
	\begin{align} 
		\begin{split}
			{\dot x}_i&=A_i x_i + B_i u_i,\\
			y_i&=C_i x_i,\label{eqsys}
		\end{split}
	\end{align}
	where $x_i\in\mathbb{R}^{n_i}$, $u_i\in\mathbb{R}^{m_i}$ and $y_i\in\mathbb{R}^{p_i}$ are the state, control input and output variables, respectively. $A_i\in\mathbb{R}^{n_i \times n_i}$, $B_i\in\mathbb{R}^{n_i \times m_i}$ and $C_i\in\mathbb{R}^{p_i \times n_i}$ are the constant state, input and output matrices, respectively.
	
		\begin{assumption} \label{asp4}
		$(A_i,B_i)$ is controllable, $(A_i,C_i)$ is detectable, and
		\begin{align}
			rank \begin{bmatrix} C_iB_i &\boldsymbol{0}_{p_i\times m_i}\\ -A_iB_i &B_i \end{bmatrix}=n_i+p_i, i\in[N].\label{rank1}
		\end{align}
	\end{assumption}
	
	\begin{lemma} [\citet{lzh2020}]\label{lemma3}
		Under Assumption \ref{asp4}, the matrix equations
		\begin{subequations}\label{eqle3}
			\begin{align}
				&B_i\Gamma_i-\Psi_i=\boldsymbol{0},\label{eqle31}\\
				&B_i\Upsilon_i-A_i\Psi_i=\boldsymbol{0},\label{eqle32}\\
				&C_i\Psi_i-I=\boldsymbol{0},\label{eqle33}
			\end{align}
		\end{subequations}
		have solutions $\Gamma_i$, $\Psi_i$, and $\Upsilon_i$.
	\end{lemma}
	
	\begin{remark}
		 The controllability and detectability in Assumption \ref{asp4} are quite standard in dealing with the problem for linear systems, and the assumption of matrix rank (\ref{rank1}) is commonly used in the design of optimizers for heterogeneous linear systems \citep{lzh2020, yu2021a}.
	\end{remark}

\subsection{Optimization Goal}
	Each agent $i$ has a local output variable $y_i$ which is limited to an output set $\mathcal{Y}_i \subseteq \mathbb{R}^{p_i}$, i.e., $y_i(t)\in\mathcal{Y}_i$.
	$f_i(t,\cdot): \mathcal{Y}_i\to \mathbb{R}$ is its local cost function
	and $g_i(t,\cdot): \mathcal{Y}_i\to \mathbb{R}^{q}$ is its private constraint function. 
	For simplicity, define $p:=\sum_{i=1}^{N}p_i$, $\mathcal{Y}:=\mathcal{Y}_1\times\dots \times \mathcal{Y}_N \subseteq \mathbb{R}^{p}$, $y:=col(y_1,\dots, y_N)\in\mathcal{Y}$, $f(t,y):=\sum_{i=1}^{N}f_{i}(t, y_i)$, and 
	$\bar g(t,y):=\sum_{i=1}^N g_{i}(t, y_i)$.
	
	This study aims to develop a distributed control law $u_i(t)$ for each agent to minimize the global cost function over a period of time $[0,T]$ with time-varying coupled inequality constraints:
	\begin{align}
		\begin{split}
			&\min_{y \in \mathcal{Y}} \int_{0}^{T}f(t,y)\,dt ,\\
			&s.t.~\bar g(t,y)\le \boldsymbol0.\label{question2}
		\end{split}
	\end{align}


	\begin{assumption} \label{asp2}
		Each set $\mathcal{Y}_i$ is convex and compact for all $i\in[N]$. For $t\in [0,T]$, functions $f_i(t,y_i)$ and $g_i(t,y_i)$ are  integrable with respect to $t$, convex with respect to $y_i$,  and Lipschitz continuous over $\mathcal{Y}_i$. The boundedness of $f_i$ and $g_i$ can be inferred, i.e., there exist constants $K_f>0$ and $K_g>0$ such that $|f_i(t,y_i)|\le K_f$ and $\|g_i(t,y_i)\|\le K_g$.
	\end{assumption}
	
	\begin{assumption}\label{asp2a}
		The local cost function $f_i(t,y_i)$ is $l$-strongly convex.
	\end{assumption}
	
	\begin{assumption}\label{asp2e}
		The set of feasible outputs $\mathcal{Y}^\dagger:=\{y: y\in\mathcal{Y}, \sum_{i=1}^{N}g_{i}(t, y_i)\le 0, t\in[0,T]\}$ is non-empty.
	\end{assumption}

	\begin{remark}
		 Assumption \ref{asp2} is reasonable since the outputs like voltage frequently have a certain range in practice. The cost functions and constraints are not required to be differentiable, which can be dealt with by using subgradients. Assumption \ref{asp2a} and \ref{asp2e} are standard \citep{yi2020,pater2020}.
	\end{remark}

\subsection{ Performance Metrics}
	Two performance indicators, network regret and network fit, are defined to evaluate the cost performance of such output trajectories. Regret is defined as
	\begin{align}
		\mathcal{R}^T:=\int_{0}^{T} \Big(f(t,y(t))-f(t,y^*)\Big)\,dt\label{regert},
	\end{align}
	where $y^*=(y_1^*,\dots,y_N^*)\in \mathcal{Y}$ represents the offline response to problem (\ref{question2}) obtained by the omniscient agent who knows all information about all agents in the period of $[0, T]$, i.e.,
	\begin{align}
		\begin{split}
		&y^*=\underset{y\in\mathcal{Y}}{\arg\min} \int_{0}^{T}f(t,y)\,dt ,\\
		&s.t. ~\bar g(t,y)\le \boldsymbol0.
		\end{split}
	\end{align}
	In general, if the regret is sublinear concerning $T$, the proposed method is deemed ``good" \citep{li2022sur}.

	Fit is defined as the projection of the cumulative constraints onto the nonnegative orthant:
	\begin{align}
		\mathcal{F}^T:=\left\|\,\left[\int_{0}^{T} \sum_{i=1}^{N}g_{i}(t, y_i) \,dt\right]_+\right\|,\label{fit}
	\end{align}
	which assesses how well the output trajectories $y(t)$ conform to the constraints.
	This definition implies that strictly viable decisions are permitted to make up for the constraints that are occasionally violated. When the outputs can be stored, like average power, then this is appropriate \citep{pater2017}.
	Similarly, the algorithm is deemed ``good" if the fit is sublinear concerning T.
	
	By $g_i(t,\cdot): \mathbb{R}^{p_i}\to \mathbb{R}^{q}$, one can define $g_{i,j}(t,\cdot): \mathbb{R}^{p_i} \to \mathbb{R}$ as the $j$th component of $g_i(t,\cdot)$, i.e., $g_i(t,\cdot) = col\left(g_{i,1}(t,\cdot), \dots, g_{i,q}(t,\cdot)\right)$. Further, define $F_j^T:=\int_{0}^{T} \sum_{i=1}^{N}g_{i,j}(t, y_i)\,dt, j=1,\dots,q$ as the $j$th component of the constraint integral. It naturally follows that 
	\begin{align}
		\mathcal{F}^T=\sqrt{\sum_{j=1}^q\left[F_j^T\right]_+^2}.
	\end{align}

	To analyze the projected dynamical system, a differentiated projection operator is defined as follows.
	
	\begin{definition} [\citep{zhang1995}]
		Let $\mathcal{S} \subseteq \mathbb{R}^n$ be a closed convex set. Then, for any $x \in \mathcal{S}$ and $v\in\mathbb{R}^n$, the projection of $v$ over set $\mathcal{S}$ at the point $x$ can be defined as
		\begin{align}
			\Pi_\mathcal{S} [x,v] = \lim\limits_{\xi\to0^+} \frac{P_\mathcal{S}(x+\xi v)-x}{\xi}.\notag
		\end{align}
	\end{definition}

	\begin{lemma} [\citep{pater2017}]
		Let $\mathcal{S} \subseteq \mathbb{R}^n$ be a closed convex set and $x_1, x_2\in \mathcal{S}$, then one has
		\begin{align}
			(x_1-x_2)^\top \Pi_{\mathcal{S}}(x_1,v)\le (x_1-x_2)^\top v, \forall v\in \mathbb{R}^n.\label{eqle2}
		\end{align}
	\end{lemma}


\section{Main Results}\label{mai}
	
	 In this section, we firstly design a controller under continuous-time settings and analyze the convergence performance of the closed-loop system. Then, an event-triggered mechanism is introduced to adapt to practical scenarios by discretizing communication and reducing communication load. Finally, the impact of communication noise on the controller's performance is considered. 
	
	\subsection{Continuous Communication}
	
	Firstly,  for the constrained optimization problem, one can construct the time-varying Lagrangian for agent $i$ as
	\begin{align}
		\mathcal{L}_i(t,y_i,\mu_i)=f_i(t,y_i)+\mu_i^\top g_i(t,y_i) - K_\mu h_i,\label{la}
	\end{align}
	where $\mu_i\in\mathbb{R}_+^{q}$ is the local Lagrange multiplier for agent $i$, $K_\mu>0$ is the preset parameter and $h_i:= \frac{1}{2}\sum_{j=1}^N a_{ij}\left\|\mu_i-\mu_j\right\|_1$ is a metric of $\mu_i$'s disagreement {\citep{liang2018}}.  
	
	For simplicity, define
	\begin{align}
		\mathcal{L}(t,y,\mu):= f(t,y)+\mu^\top g(t,y) - K_\mu h(\mu),\label{lag}
	\end{align}
	where $g(t,y)=col(g_1(t,y_1), \dots, g_N(t,y_N))$, $\mu=col(\mu_1, \dots, \mu_N)$, and $h(\mu):=\sum_{i=1}^N h_i(\mu)$. It can be easily verified that  $\mathcal{L}(t,y,\mu)=\sum_{i=1}^{N}\mathcal{L}_i(t,y_i,\mu_i)$.
	
	Given that $f_i(t,\cdot)$, $g_i(t,\cdot)$ are convex and $\mu_i \ge \boldsymbol 0_q$,
	the Lagrangian is convex with respect to $y_i$. Let us
	denote by $\partial_{y_i}\mathcal{L}(t,y,\mu)$ a subgradient of $\mathcal{L}$ with respect to $y_i$, i.e.,
	\begin{align}
		\partial_{y_i}\mathcal{L}(t,y,\mu)\in \partial f_i(t,y_i)+\mu_i^\top \partial g_i(t,y_i).\label{lax}
	\end{align}

	Similarly, the Lagrangian is concave with respect to $\mu_i$ and its subgradient is given by
	\begin{align}
		\partial_{\mu_i}\mathcal{L}(t,y,\mu)&\in g_i(t,y_i) - K_\mu\sum_{j=1}^N a_{ij}\sign{\mu_i-\mu_j}.\label{lamu}
	\end{align}

	Design the following distributed closed-loop observer
	\begin{subequations}\label{ob1}
		\begin{align}
			&\dot{\hat{x}}_i=A_i \hat{x}_i + B_iu_i + H_i(y_i -\hat{y}_i), \label{ob1a}\\
			&\hat{y}_i=C_i\hat{x}_i,\label{ob1b}
		\end{align}
	\end{subequations}
	where $\hat{x}_i$ is an estimate of $x_i$, $H_i$ is the feedback matrix to be determined later, and $\hat{y}_i=C_i\hat{x}_i$ is the estimated output.
	
	 Based on the idea of treating optimization algorithms as feedback controllers \citep{hau2021}, an observer-based controller, following the modified Arrow-Hurwicz algorithm \citep{AH1958} for the $i$th agent, is proposed as
	\begin{subequations}\label{eqpi1}
		\begin{align}
			u_i&=-K_{i} \hat{x}_i+\Gamma_i\dot{\eta}_i - (\Upsilon_i - K_i\Psi_i)\eta_i, \label{eqpi1a}\\
			\dot{\eta}_i&=\Pi_{\mathcal{Y}_i}[y_i, -\varepsilon \partial_{y_i}\mathcal{L}(t,y,\mu)],\label{eqpi1b}\\
			\dot{\mu_i}&=\Pi_{\mathbb{R}_+^q}[\mu_i,\varepsilon \partial_{\mu_i}\mathcal{L}(t,y,\mu)], \label{eqpi1c}
		\end{align}
	\end{subequations}
	where $\varepsilon>0$ is the step size, $\eta_i\in\mathbb{R}^{p_i}$ is the auxiliary states, $K_i$ is the feedback matrix to be determined later, $\Gamma_i$, $\Upsilon_i$ and $\Psi_i$ are feedback matrices that are the solutions of (\ref{eqle3}), and the initial values are set as $\eta_i(0){=}\boldsymbol{0}$ and $\mu_i(0){=}\boldsymbol{0}$.

	Substituting the controller (\ref{eqpi1}) into the system (\ref{eqsys}), the closed-loop dynamics of the whole system is
	
	\begin{subequations}\label{eqpi2}
		\begin{align}
			\dot x&{=}Ax-BK\hat x-(A{-}B K)\Psi\eta{+}B\Gamma \dot{\eta}, \label{eqpi2a}\\
			\dot{\eta}&{=}\Pi_{\mathcal{Y}}[y, -\varepsilon \partial_{y}\mathcal{L}(t,y,\mu)],\label{eqpi2b}\\
			\dot{\mu}&{=}\Pi_{\mathbb{R}_+^{Nq}}[\mu,\varepsilon \partial_{\mu}\mathcal{L}(t,y,\mu)] ,\label{eqpi2c}\\
			y&{=}Cx,\label{eqpi2d}
		\end{align}
	\end{subequations}
	where $x {=} col(x_1,\dots,x_N)$, $\hat x {=} col(\hat x_1,\dots,\hat x_N)$, $\eta {=} col(\eta_1,\dots,\eta_N)$, $A {=} diag(A_1,\!\dots\!,A_N)$, $B {=} diag(B_1, \!\dots\!, B_N)$, $C {=} diag(C_1, \!\dots\!, C_N)$, $K = diag(K_{\!1}, \!\dots\!,\! K_{\!N})$, $\Psi {=} diag(\Psi_{\!1}, \!\dots\!,\!\! \Psi_{\!N}),\!$ and $\Gamma{=}diag(\Gamma_{\!1}, \!\dots\!,\! \Gamma_{\!N}\!)$.
	
	\begin{remark}
		Since the algorithms presented in this paper involve the use of sign functions which exhibit discontinuity, the solutions of system (\ref{eqsys}) with designed controllers are considered in the sense of Filippov \citep{dis1988}.
	\end{remark}
	
	For the subsequent analysis, consider the following energy function with any $\tilde{y}\in\mathcal{Y}$ and $\tilde{\mu}\in\mathbb{R}_+^{Nq}$ :
	\begin{align}
		V_{(\tilde{y}, \tilde{\mu})}(x,\mu)=V_1 + V_2 + V_3,\label{v1}
	\end{align}
	where
	\begin{subequations}
		\begin{align}
			&V_1=\frac{1}{2}\|Cx-\tilde{y}\|^2,\label{v1a}\\
			&V_2=\frac{1}{2}\|\mu-\tilde{\mu}\|^2,\label{v1b}\\
			&V_3=(x-\Psi\eta)^\top P\, (x-\Psi\eta),
		\end{align}
	\end{subequations}
	and $P$ is some positive definite matrix.
	
	The following lemma establishes the relationship between the above energy function and time-varying Larangian (\ref{la}) along the system dynamics (\ref{eqpi1}).

	\begin{lemma}\label{lemma4}
		If Assumptions \ref{asp1}-\ref{asp2e} hold and $\tilde{\mu}:= \boldsymbol1_N \otimes \gamma, \forall \gamma\in \mathbb{R}_+^q$, then for any $T\ge0$ the linear multi-agent system (\ref{eqsys}) under control protocol (\ref{eqpi1}) satisfies
		\begin{align}
			\int_{0}^{T} \Big( \mathcal{L}(t,y(t),\tilde{\mu})-\mathcal{L}(t,\tilde{y},\mu(t)) \Big) \,dt\le\frac{V_{(\tilde{y}, \tilde{\mu})}(x(0), \boldsymbol{0})}{\varepsilon}.\label{l3}
		\end{align}
	\end{lemma}
	
	\begin{proof}
		See Appendix $A$.  In the proof, we use the nonsmooth penalty function $h_i$ to drive $\mu_i$ to reach consensus. Additionally, by using the convexity of $y_i$ in $\mathcal{L}$ and the concavity of $\mu_i$ in $\mathcal{L}$, the relationship between Lagrangian $\mathcal{L}$ and the energy function $V$ is established.	
	\end{proof}

	We now state the main result about the regret and fit bounds under controller (\ref{eqpi1}) with continuous communication.
	
	\begin{theorem}\label{th1}
		Suppose that Assumptions \ref{asp1}-\ref{asp2e} hold. Let $K_i$, $H_i$, $i\in[N]$ be such that $A_i-B_iK_i$ and $A_i-H_iC_i$ are Hurwitz. Then for any $T\ge0$, under controller (\ref{eqpi1}) with $\varepsilon>0$ and $K_\mu \ge NK_g$, the following regret and fit bounds hold
		\begin{align}
			&\mathcal{R}^T\le \frac{\|y(0)\!-\!y^*\|^2 + x(0)^\top P x(0)}{2\varepsilon};\notag\\
			&\mathcal{F}^T {\le} \frac{\sqrt{N}\|y(0){-}y^*\| {+}\sqrt{2N x(0)^{\!\top} \!P\! x(0)}}{\varepsilon}  {+}  2N\sqrt{\frac{K_f}{\varepsilon}}\sqrt{T}.\notag
		\end{align}
	\end{theorem}
	
	\begin{proof}	
		See Appendix $B$.  In the proof, the regret and fit bounds are obtained through proper  selection of the parameters of Lemma \ref{lemma4}. Furthermore, the connectivity of the graph is utilized to eliminate unnecessary terms in the proof.
	\end{proof}
	
	\begin{remark}\label{remark2}		
		Theorem \ref{th1} shows that $\mathcal{R}^T=\mathcal{O}(1)$ and $\mathcal{F}^T = \mathcal{O}(\sqrt{T})$ under continuous communication, which match the best-known results in the centralized setting \citep{pater2017}.
		In comparison, for single-integrator multi-agent systems, explicit bounds on both the regret and fit with a sublinear growth are obtained in \citet{yi2020,li2022}  i.e., $\mathcal{R}^T=\mathcal{O}(T^{\max\{\kappa,1-\kappa\}})$, $\mathcal{F}^T = \mathcal{O}(T^{\max\{\kappa,1-\kappa\}})$ in \citet{yi2020} and $\mathcal{R}^T= \mathcal{O}(T^{\max\{\kappa,1-\kappa\}})$, $\mathcal{F}^T = \mathcal{O}(T^{\max\{\frac{1}{2}+\frac{\kappa}{2},1-\frac{\kappa}{2}\}})$ in \citet{li2022} for $\kappa\in(0,1)$. Theorem \ref{th1} achieves a lower regret bound than \citet{yi2020,li2022} under more complex system dynamics. 
		 However, \citet{yi2020,li2022} have also established the bounds of dynamic regret and fit, which may be one of our future works. In addition, the designed controller (\ref{eqpi1}) needs to know the total number of agents $N$ and the upper bound $G$ of all $f_i$, which, however, can be estimated in a distributed manner \citep{oli2017} within a finite time.
	\end{remark}

	{\em Extension to the Identical Output Case.}
	In problem (\ref{question2}), the output of each agent can be different, which only needs to satisfy the coupling constraints. When considering the distributed online consensus optimization with coupling inequality constraints \citep{lxx2020c}, the problem can be modeled as
	\begin{align}
		\begin{split}
			\min_{y \in \mathcal{Y}} &\int_{0}^{T} f(t,y)\,dt ,\\
			s.t.\  &y_i=y_j, \forall i,j\in[N],\\
			&\sum_{i=1}^{N}g_{i}(t, y_i)\le \boldsymbol0.\label{question3}
		\end{split}
	\end{align}
	Correspondingly, as in \citet{pater2020, ydm2022}, the individual regret of the $i$-th agent is defined as
	\begin{align}
		\mathcal{R}_i^T:=\int_{0}^{T} \Big(\sum_{j=1}^N f_j(t,y_i)-f(t,y^*)\Big)\,dt.
	\end{align}
	For this problem, the time-varying Lagrangian for agent $i$ can be modified as
	\begin{align}
		\widehat{\mathcal{L}}_i(t,y_i,\mu_i)=f_i(t,y_i)+\mu_i^\top g_i(t,y_i) - K_\mu h_i + K_y \chi_i,\label{la2}
	\end{align}
	where $K_y>0$ is the preset parameter and $\chi_i:= \frac{1}{2}\sum_{j=1}^N a_{ij}\left\|y_i-y_j\right\|_1$ as a metric of $y_i$'s disagreement. Define $\widehat{\mathcal{L}}(t,y,\mu) := \sum_{i=1}^{N}\widehat{\mathcal{L}}_i(t,y_i,\mu_i)$ and $\chi(t):=\sum_{i=1}^{N} \chi_i$.	
	Accordingly, the controller (\ref{eqpi1b}) is modified as
	\begin{align}
		\dot{\eta}_i&=\Pi_{\mathcal{Y}_i}[y_i, -\varepsilon \partial_{y_i}\widehat{\mathcal{L}}(t,y,\mu)],
	\end{align}
	where $\partial_{y_i}\widehat{\mathcal{L}}(t,y,\mu)\in \partial f_i(t,y_i) + K_y\sum_{j=1}^N a_{ij}\sign{y_i-y_j} +\mu_i^\top \partial g_i(t,y_i)$ is a subgradient of $\widehat{\mathcal{L}}$ with respect to $y_i$. Lemma \ref{lemma3} still holds by using this control protocol. Then, the regret of the $i$-th agent can be rewritten as
	\begin{align}
		\mathcal{R}_i^T{=}
		&\int_{0}^{T} \!\!\left(\widehat{\mathcal{L}}(t,y,\boldsymbol{0}_{Nq}) {-} \widehat{\mathcal{L}}(t,y^*,\mu)\right)dt
		{+}\int_{0}^{T} \!\!\mu^\top g(t,y^*)\, dt\notag\\
		&+\!\!\int_{0}^{T}\!\!\!\!\! \sum_{j=1,j\neq i}^N\!\! \big(f_j(y_i){-}f_j(y_j)\big)dt  {+}\!\!\int_{0}^{T}\!\!\!\big(K_y\chi(y){-}K_\mu h(\mu)\big) dt.
	\end{align}
	Additionally, assume that there exists a positive constant $K_{\partial f}$ such that $\|\partial f_i(t,y_i)\|\le K_{\partial f}$ for any $y_i\in\mathcal{Y}_i$ and $i\in[N]$.
	Since one has $f_j(y_j)\ge f_j(y_i) - K_{\partial f}\|y_j-y_i\|_1$, by following steps similar to those in the proof of Theorem \ref{th1} and additionally choosing $K_y\ge NK_{\partial f}$, one has $\mathcal{R}_i^T=\mathcal{O}(1)$ and $\mathcal{F}^T = \mathcal{O}(\sqrt{T})$. In summary, the  controller in this paper can be easily extended to the scenario with an identical output for all agents.	

	\subsection{Event-triggered Communication}
	The above continuous-time control law, which requires each agent to know the real-time Lagrange multipliers of neighbors, may cause excessive communication overhead. 
	In this section, an event-triggered protocol is proposed to avoid continuous communication. 
	
	For agent $i$, suppose that $t_i^l$ is its $l$th communication instant and $\{t_i^1, \dots, t_i^l, \dots\}$ is its communication instant sequence. Define $\hat{\mu}_i(t)\!:=\! \mu_i(t_i^l), ~\forall t\!\in\! [t_i^l,t_i^{l+1})$ as the available information of its neighbors and $e_i\!:=\!\hat \mu_i(t)-\mu_i(t)$ as the measurement error. It can be known  that $e_i=0$ at any instant  $t_i^l$.
	
	As defined in \citet{sz2020}, a triggering is Zeno for agent $i$ if $\lim_{l\to\infty} t_i^l=t_i^\infty$ for some finite $t_i^\infty$, i.e., an infinite amount of events will be triggered in finite time $t_i^\infty$, which is unachievable for a physical system because of its limited  hardware frequency.
	It is also showed in this section that the closed-loop systems have no Zeno behavior.

	An event-triggered control law is proposed as
	\begin{subequations}\label{eqpi4}
		\begin{align}
			u_i&=-K_{i} \hat{x}_i+\Gamma_i\dot{\eta}_i - (\Upsilon_i - K_i\Psi_i)\eta_i, \label{eqpi4z}\\
			\dot{\eta}_i&=\Pi_{\mathcal{Y}_i}[y_i, -\varepsilon \partial_{y_i}\mathcal{L}(t,y,\mu)],\label{eqpi4a}\\
			\dot{\mu_i}&=\Pi_{\mathbb{R}_+^q}[\mu_i,\varepsilon g_i(t,y_i)-2\varepsilon K_\mu \sum_{j=1}^N a_{ij}\sign{\hat\mu_i-\hat\mu_j}], \label{eqpi4b}
		\end{align}
	\end{subequations}
	where $\hat\mu_i$ is the most recent broadcast value of $\mu_i$.
	Note that $0$ is chosen for the sign function in (\ref{eqpi4b}) when its argument is zero.
	
	The communication instant is chosen as
	\begin{align}
		t_i^{l+1}\!:=\!\inf_{t>t_i^l}\!\Big\{t\Big|\|e_i(t)\|{\ge}\frac{1}{6N\!\sqrt{q}}\!\sum_{j=1}^N \!a_{i\!j} \|\hat\mu_i{-}\hat\mu_j\|_{1} {+} \frac{\sigma e^{-\iota t}}{3N^{2}K_{\!\mu}\!\sqrt{q}}\Big\},\label{tau2}
	\end{align}
	where $\sigma$ and $\iota$ are prespecified positive real numbers. 
	
	At time $t=0$, each agent $i\in[N]$ initializes $\hat\mu_i=\mu_i(0)$. Then, at each time $t$, if the triggering condition (\ref{tau2}) is satisfied, agent $i$ then broadcasts the information $\hat\mu_i=\mu_i(t)$ to its neighbors and updates the control law (\ref{eqpi4b}). Meanwhile, agent $i$ updates the controller (\ref{eqpi4b}) whenever it receives new information $\hat\mu_j$ from its neighbors.

	The following lemma is a modification of Lemma \ref{lemma4} under event-triggered communication.

	\begin{lemma}\label{lemma6}
		If Assumptions \ref{asp1}-\ref{asp2e} hold and $\tilde{\mu}= \boldsymbol1_N \otimes \gamma, \forall \gamma\in \mathbb{R}_+^q$, then for any $T\ge0$ the linear multi-agent system (\ref{eqsys}) under control protocol (\ref{eqpi4}) satisfies
		\begin{align}
			\int_{0}^{T} \Big(\mathcal{L}(t,y(t),\tilde{\mu}){-}\mathcal{L}(t,\tilde{y},\mu(t))\Big)dt{\le}\frac{V_{(\tilde{y}, \tilde{\mu})}(y(0), \boldsymbol{0})}{\varepsilon} {+} \frac{\sigma}{\iota}.
		\end{align}
	\end{lemma}
	\begin{proof}
		See Appendix $C$.
	\end{proof}
	
	We now state the main result about the regret and fit bounds with event-triggered communication controller (\ref{eqpi4}).
	
	\begin{theorem}\label{th3}
		Suppose that Assumptions \ref{asp1}-\ref{asp2e} hold. Let $K_i$, $H_i$, $i\in[N]$ be such that $A_i-B_iK_i$ and $A_i-H_iC_i$ are Hurwitz. Then for any $T\ge0$, under controller (\ref{eqpi4}) with $\varepsilon>0$ and $K_\mu \ge NK_g$, the following regret and fit bounds hold:
		\begin{align}
			\mathcal{R}^T\le &\frac{\|y(0)\!-\!y^*\|^2 + x(0)^\top P x(0)}{2\varepsilon} +\frac{\sigma}{\iota};\notag\\
			\mathcal{F}^T\le &\frac{\sqrt{N}\|y(0)-y^*\|}{\varepsilon}  +
			\frac{\sqrt{2N x(0)^{\!\top} \!P\! x(0)}}{\varepsilon}
			+\sqrt{\frac{2N\sigma}{\varepsilon\iota}} \notag\\
			&+ 2N\sqrt{\frac{K_f}{\varepsilon}}\sqrt{T}.\notag
		\end{align}
		Moreover, under the event triggering condition (\ref{tau2}), the close-loop system does not exhibit the Zeno behavior.
	\end{theorem}
	
	\begin{proof}
		See Appendix $D$. The proofs of regret and fit bounds of Theorem \ref{th3} are similar to those of Theorem \ref{th1}, except that Lemma \ref{lemma6} is used instead of Lemma \ref{lemma4}. The exponential decay term $e^{-\iota t}$ is introduced in (\ref{tau2}) to exclude Zeno behavior.	
	\end{proof}
	
	\begin{remark}\label{remark3}
		Theorem \ref{th3} means that $\mathcal{R}^T=\mathcal{O}(1)$ and $\mathcal{F}^T = \mathcal{O}(\sqrt{T})$ still hold even under event-triggered communication. The bounds of regret and fit are determined by the communication frequency. Generally speaking, decreasing $\sigma$ and increasing $\iota$ will achieve smaller bounds on regret and fit, but meanwhile increase the communication frequency, which results in a tradeoff between them.	
	\end{remark}
	
	\begin{remark}
		The event-triggered communication mechanism has been widely used in various distributed optimization algorithms \citep{lzh2020, yy2022, yt2022} because of its  high flexibility and scalability \citep{ now2019}. But to the best of our knowledge, existing works on distributed online optimization have not considered the event-triggered mechanism. The analysis of Theorem \ref{th3} shows that using event-triggering control law can guarantee the performance of the  systems while reducing the communication overhead.
	\end{remark}
	

	
	\subsection{Continuous Communication with Noise}\label{nono}
	
	In  controller (\ref{eqpi1}), the sign functions can be roughly regarded as the relative directions of the intermediate variables since $\sign{\mu_i(t)-\mu_j(t)}=\sign{\frac{\mu_i(t)-\mu_j(t)}{\|\mu_i(t)-\mu_j(t)\|}}$.
	In practice, measurements of the relative direction of neighbors may be inaccurate due to communication noise or sensor inaccuracies. To examine the effect of this inaccuracy on the direction of relative state measurements, we replace $\sign{\mu_i-\mu_j}$ by $\sign{\frac{\mu_i-\mu_j}{\|\mu_i-\mu_j\|}+\epsilon_{ij}(t)}= \sign{\mu_i-\mu_j+\epsilon_{ij}(t)\|\mu_i-\mu_j\|}$, where $\epsilon_{ij}(t)\in \mathbb{R}^q$ is the time-varying communication noise from agent $j$ to $i$. 
	It is noteworthy that the noise $\epsilon_{ij}(t)$ is incurred when the {\em direction} of the relative state of the neighboring agent $j$ is measured by agent $i$, and thereby is added to the normalized relative state, as done in \citet{cao2021}.
	
	Similar to (\ref{regert}) and (\ref{fit}), the regret and fit are defined in the probabilistic sense as
	\begin{align}
		&\overline{\mathcal{R}}^T:=\mathbb{E}\left[\int_{0}^{T} \Big(f(t,y(t))-f(t,y^*)\Big)\,dt\right];\label{regertn}\\
		&\overline{\mathcal{F}}^T:=\left\|\,\mathbb{E}\left[\,\Big[\int_{0}^{T} \sum_{i=1}^{N}g_{i}(t, y_i) \,dt\Big]_+\right]\,\right\|.\label{fitn}
	\end{align}

	\begin{assumption}\label{asp5}
		The noises $\epsilon_{ij}(t)$ are independent and identically distributed. And their common probability density function,  denoted by $q(\epsilon)$, is symmetric.
	\end{assumption}
	

		The time-varying Lagrangian with noise is constructed as
		\begin{align}
			\mathcal{\tilde{L}}(t,y,\mu)=\sum_{i=1}^{N}f_i(t,y_i)+\mu^\top g(t,y) - K_\mu \tilde{h}(\mu),\label{lan}
		\end{align}
		where $\tilde{h}(\mu):= \frac{1}{2}\sum_{i=1}^N\sum_{j=1}^N a_{ij}\big\|\mu_i-\mu_j +\epsilon_{ij}(t)\|\mu_i(t)-\mu_j(t)\|\big\|_1$ and $\mu(t)$  is the trajectory generated by controller (\ref{eqpi7}).

		\begin{lemma}\label{lemma5}
			If Assumptions \ref{asp1}-\ref{asp5} hold and $\tilde{\mu}:= \boldsymbol1_N \otimes \gamma, \forall \gamma\in \mathbb{R}_+^q$, then for any $T\ge0$ the linear multi-agent system (\ref{eqsys}) under control protocol (\ref{eqpi1}) satisfies
			\begin{align}
				\int_{0}^{T} \mathbb{E}\big\{ \mathcal{\tilde{L}}(t,y(t),\tilde{\mu}) - \mathcal{\tilde{L}}(t,\tilde y,\mu(t)) \big\} \,dt\le\frac{V_{(\tilde{y}, \tilde{\mu})}(y(0), \boldsymbol{0})}{\varepsilon}.\label{ln1}
			\end{align}
		\end{lemma}
	
	\begin{proof}
		See Appendix $E$.
	\end{proof}

	A continuous control law using noisy signs is proposed as
	\begin{subequations}\label{eqpi7}
		\begin{align}
			u_i=&-K_{i} \hat{x}_i+\Gamma_i\dot{\eta}_i - (\Upsilon_i - K_i\Psi_i)\eta_i, \label{eqpi7z}\\
			\dot{\eta}_i=&\Pi_{\mathcal{Y}_i}[y_i, -\varepsilon \partial_{y_i}\mathcal{L}(t,y,\mu)],\label{eqpi7a}\\
			\dot{\mu_i}=&\Pi_{\mathbb{R}_+^q}\Big[\mu_i,\varepsilon g_i(t,y_i){-}\varepsilon K_\mu \sum_{j=1}^N a_{ij} \sign{\mu_i{-}\mu_j{+}\epsilon_{ij}(t)\|\mu_i{-}\mu_j\|}\Big], \label{eqpi47}
		\end{align}
	\end{subequations}
	where $\epsilon_{ij}(t)$ is the time-varying measurement noise from agent $j$ to $i$. Note that $0$ is chosen for the sign function in (\ref{eqpi47}) when its argument is zero.

	We now state the main result about the regret and fit bounds of controller (\ref{eqpi7}) with noisy communications.
	
	\begin{theorem}\label{th7}
		Suppose that Assumptions \ref{asp1}-\ref{asp5} hold. Let $K_i$, $H_i$, $i\in[N]$ be such that $A_i-B_iK_i$ and $A_i-H_iC_i$ are Hurwitz. Then for any $T\ge0$, under controller (\ref{eqpi7}) with $\varepsilon>0$ and $K_\mu \ge N^2K_g$, if the noises satisfy $\mathbb{E}\left[\|\epsilon_{ij}(t)\|_1\right]\le \frac{1}{2}-\frac{N^2K_g}{2K_\mu}$, the following regret and fit bounds hold
		\begin{align}
			&\overline{\mathcal{R}}^T\le 
			\frac{\|y(0)\!-\!y^*\|^2 + x(0)^\top P x(0)}{2\varepsilon};\notag\\
			&\overline{\mathcal{F}}^T\le \frac{\sqrt{N}\|y(0){-}y^*\| {+}\sqrt{2N x(0)^{\!\top} \!P\! x(0)}}{\varepsilon}  {+}  2N\sqrt{\frac{K_f}{\varepsilon}}\sqrt{T}.\notag	
		\end{align}
	\end{theorem}

	\begin{proof}
		See Appendix $F$.	
	\end{proof}

	\begin{remark}
		To the best of our knowledge, there are few existing works on distributed online optimization that take into account measurement noise, except for \citet{cao2021}, where the case of global feasible set constraints is considered. This paper further considers the problem with  local set constraints and coupled inequality constraints.
	\end{remark}

\section{Simulation}\label{sim}

		The local objective functions are time-varying quadratic functions as follows:
		\begin{align*}
			f_1(t,y_1)=&2(y_{1,a}{-}2\cos{t}{-}1)^2{+}2(y_{1,b}{-}\cos{1.5t}{-}1.5)^2;\\
			f_2(t,y_2)=&(y_{2,a}{-}\cos{2t}{-}1)^2{+}2(y_{2,b}{-}2\cos{1.7t}{-}3)^2;\\
			f_3(t,y_3)=&3(y_{3,a}{-}\cos{2t}{-}3)^2{+}(y_{3,b}{-}\cos{t}{-}1)^2;\\
			f_4(t,y_3)=&(y_{4,a}{-}3\cos{t}{-}2)^2{+}3(y_{4,b}{-}\cos{2t}{-}2)^2;\\
			f_5(t,y_5)=&(y_{5,a}{-}\cos{1.5t}{-}1.2)^2{+}(y_{5,b}{-}3\cos{1.5t}{-}1)^2\\
			&{+}2(y_{5,c}{-}\cos{2t}{-}3)^2;\\
			f_6(t,y_6)=&0.5(y_{6,a}{-}\cos{2t}{-}1)^2{+}2(y_{6,b}{-}2\cos{1.2t}{-}1)^2\\
			&{+}2(y_{6,c}{-}\cos{t}{-}1)^2.
		\end{align*}
		
		The feasible set of output variables $\mathcal{Y}\in[-1,6]^{12}$. The time-varying constraints of $6$ agents are defined as
		\begin{align*}
			g_1(t,y_1)=&(0.3\sin{15t}{+}1.7)y_{1,a}{+}(0.2\sin{10t}{+}1.8)y_{1,b}{-}1;\\
			g_2(t,y_2)=&(0.4\sin{20t}{+}1.6)y_{2,a}{+}(0.4\sin{20t}{+}1.6)y_{2,b}{-}2;\\
			g_3(t,y_3)=&(0.5\sin{10t}{+}1.5)y_{3,a}{+}(0.6\sin{25t}{+}1.4)y_{3,b}{-}3;\\
			g_4(t,y_4)=&(0.6\sin{15t}{+}1.4)y_{4,a}{+}(0.8\sin{15t}{+}1.2)y_{4,b}{-}4;\\
			g_5(t,y_5)=&(0.7\sin{10t}{+}1.3)y_{5,a}{+}(0.5\sin{10t}{+}1.5)y_{5,b}\\
			&{+}(0.3\sin{15t}{+}1.7)y_{5,c}{-}5;\\
			g_6(t,y_6)=&(0.8\sin{15t}{+}1.2)y_{6,a}{+}(0.6\sin{25t}{+}1.4)y_{6,b}\\
			&{+}(0.4\sin{20t}{+}1.6)y_{6,c}{-}6.
		\end{align*}
		The above constraint selection ensures that $\boldsymbol{0}$ must be a feasible output for all $t\ge 0$.
		
		\begin{example}
			Consider a heterogeneous multi-agent system composed of $6$ agents described by (\ref{eqsys}), where
			\begin{flalign}
				&\ x_i\in\begin{cases}
					\mathbb{R}^2 &i=1,2,3,4\\\mathbb{R}^3 &i=5,6
				\end{cases}, y_i\in\begin{cases}
					\mathbb{R}^2 &i=1,2,3,4\\\mathbb{R}^3 &i=5,6
				\end{cases},\notag
			\end{flalign}
			\begin{align}
				\begin{aligned}
					A_{1,2}&=\begin{bmatrix} 1 &0\\ 0 &2 \end{bmatrix}, & A_{3,4}&=\begin{bmatrix} 0 &2\\ -1 &1 \end{bmatrix}, &
					A_{5,6}&=\begin{bmatrix} 2 &1 &0\\ 0 &1 &1\\ 1 &0 &2 \end{bmatrix}, \\
					B_{1,2}&=\begin{bmatrix} 0 &1\\ 1 &3 \end{bmatrix}, 	& B_{3,4}&=\begin{bmatrix} 2 &1\\ 1 &0 \end{bmatrix}, & B_{5,6}&=\begin{bmatrix} 1 &0 &0\\ 0 &1 &0\\ 0 &0 &1 \end{bmatrix},\\
					C_{1,2}&=\begin{bmatrix} 2 &0\\ 0 &1 \end{bmatrix}, & C_{3,4}&=\begin{bmatrix} 2 &1\\ -1 &0 \end{bmatrix}, &
					C_{5,6}&=\begin{bmatrix} 3 &0 &0\\ 0 &1 &0\\ 0 &1 &2 \end{bmatrix}.
				\end{aligned}\notag
			\end{align}
			
			\begin{figure}[tbp]
				\centering
				\includegraphics[width=0.45\textwidth]{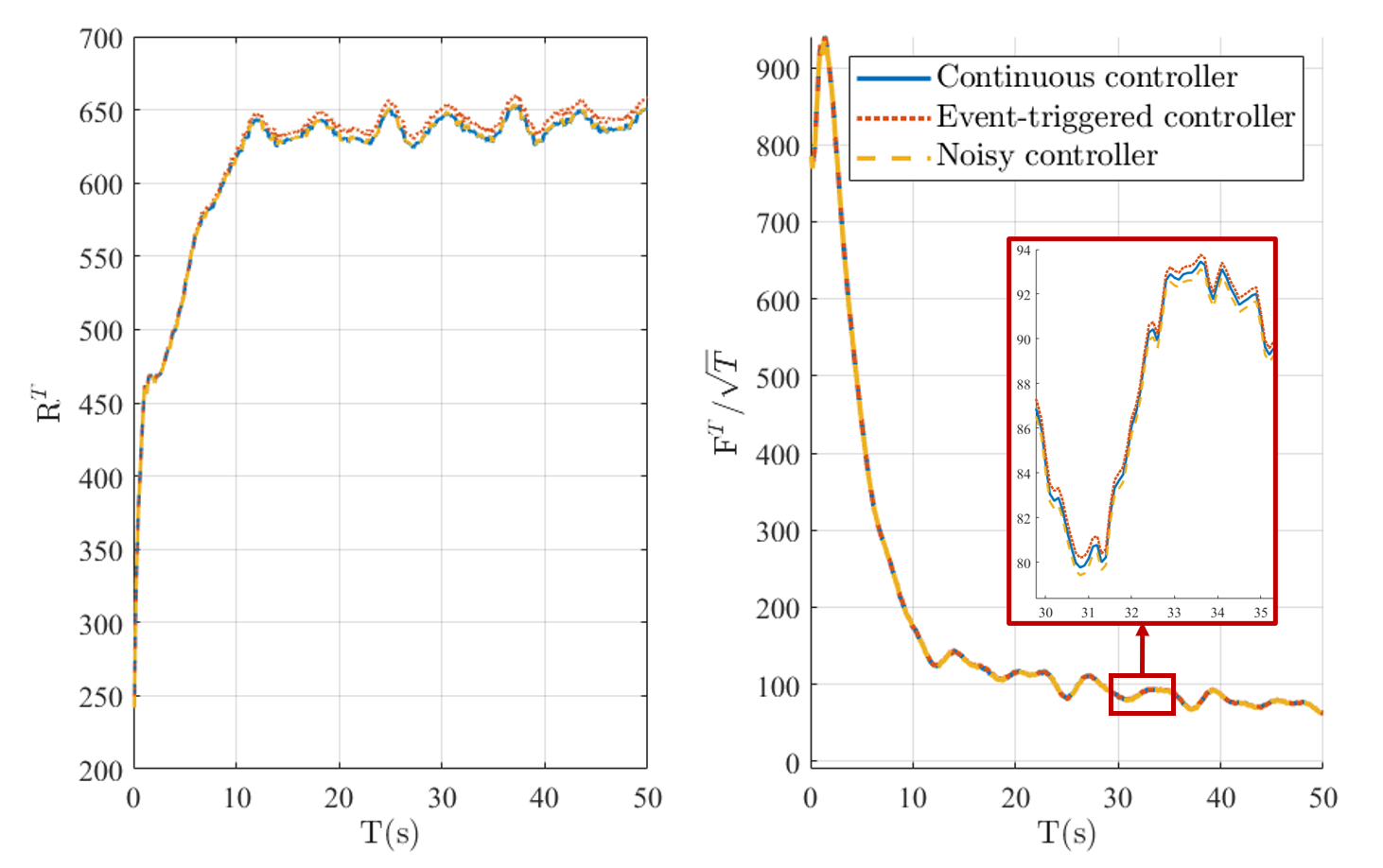} 
				\caption{ Evolution of $\mathcal{R}^T$ and $\mathcal{F}^T/\sqrt{T}$ with $3$ controllers.}
				\label{fig1}
			\end{figure}
		
		The communication network among these agents is  a ring topology, which is undirected and connected.
		The above settings satisfy Assumptions \ref{asp1}$-$\ref{asp2e}.

		For the numerical example, the selection of feedback matrices  \{$\Gamma_{i}$, $\Psi_{i}$, $\Upsilon_{i}$, $K_{i}$, $H_{i}$\}  is based on (\ref{eqle3}) and Theorem \ref{th1}. 	The the step size is set as $\varepsilon=0.1$. The initial values $x_i(0)$ are randomly selected in $[-5, 5]$, $\eta(0)=\boldsymbol{0}$, and $\mu(0)=\boldsymbol{0}$.
				
			\begin{figure}[tbp]
			\centering
			\includegraphics[width=0.4\textwidth]{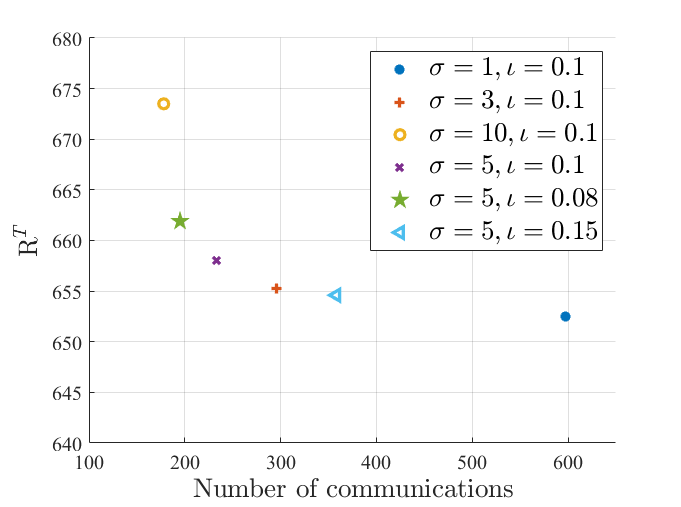} 
			\caption{ The regret with different $\sigma$ and $\iota$.}
			\label{fig2}
		\end{figure}
		
		 Figure \ref{fig1} illustrates that the continuous-time controller, event-triggered controller and controller with noise achieve constant regret bounds and sublinear fit bounds. The results are consistent with those established in Theorems \ref{th1}, \ref{th3} and \ref{th7}. Figure 2 displays the relationship between the total communication number and regret of the event triggered controller after $50$ seconds when $\sigma$ and $\iota$ take different values, showing the trade-off between communication frequency and regret value. 
\end{example}	
	
\begin{figure}[htbp]
	\centering
	\includegraphics[width=0.4\textwidth]{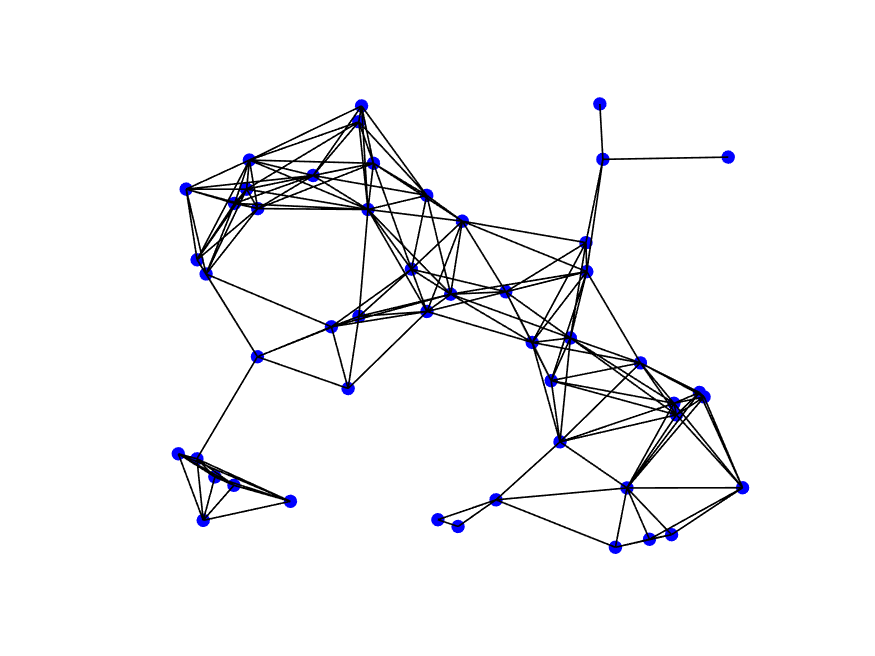} 
	\caption{Network of N=50 agents.}
	\label{pic21}
\end{figure}
\begin{figure}[htbp]
	\centering
	\includegraphics[width=0.45\textwidth]{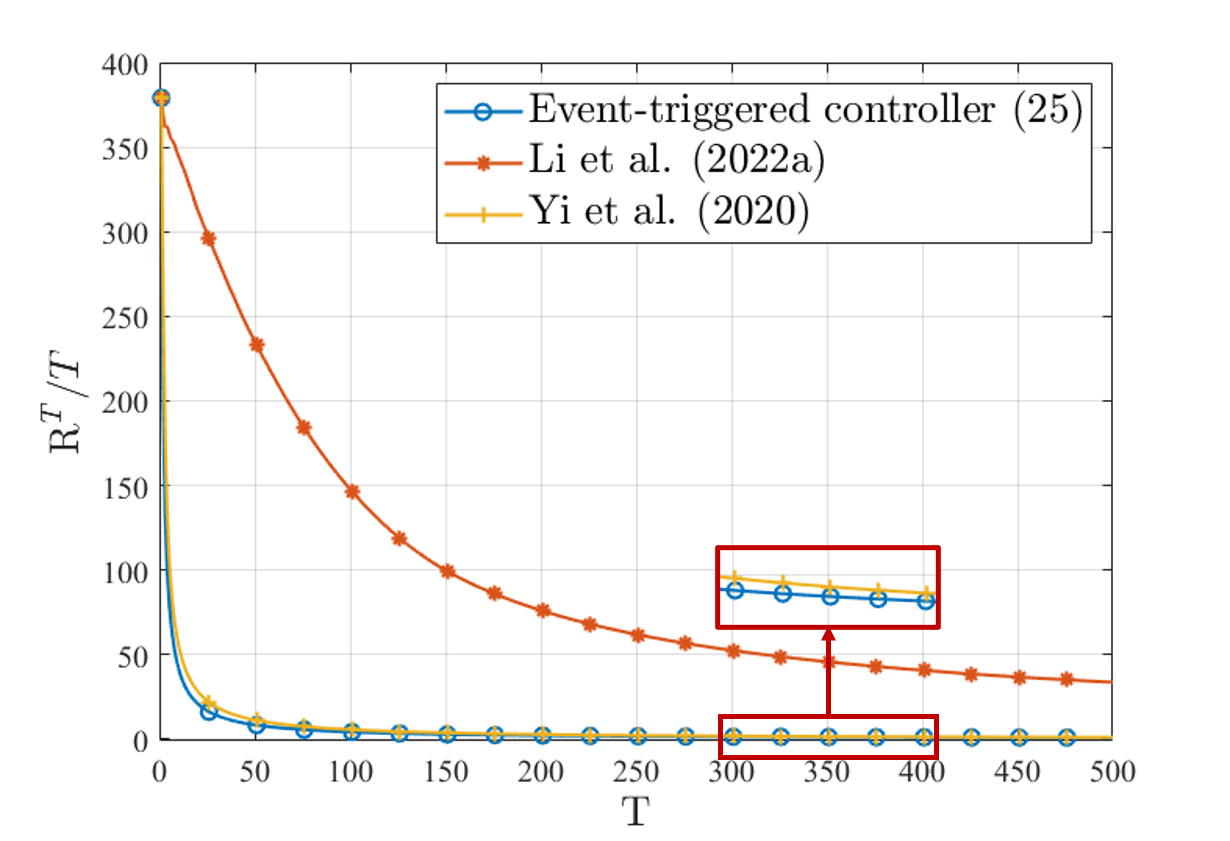} 
	\caption{ Evolution of $\mathcal{R}^T/T$.}
	\label{fig3}
\end{figure}

\begin{figure}[htbp]
	\centering
	\includegraphics[width=0.45\textwidth]{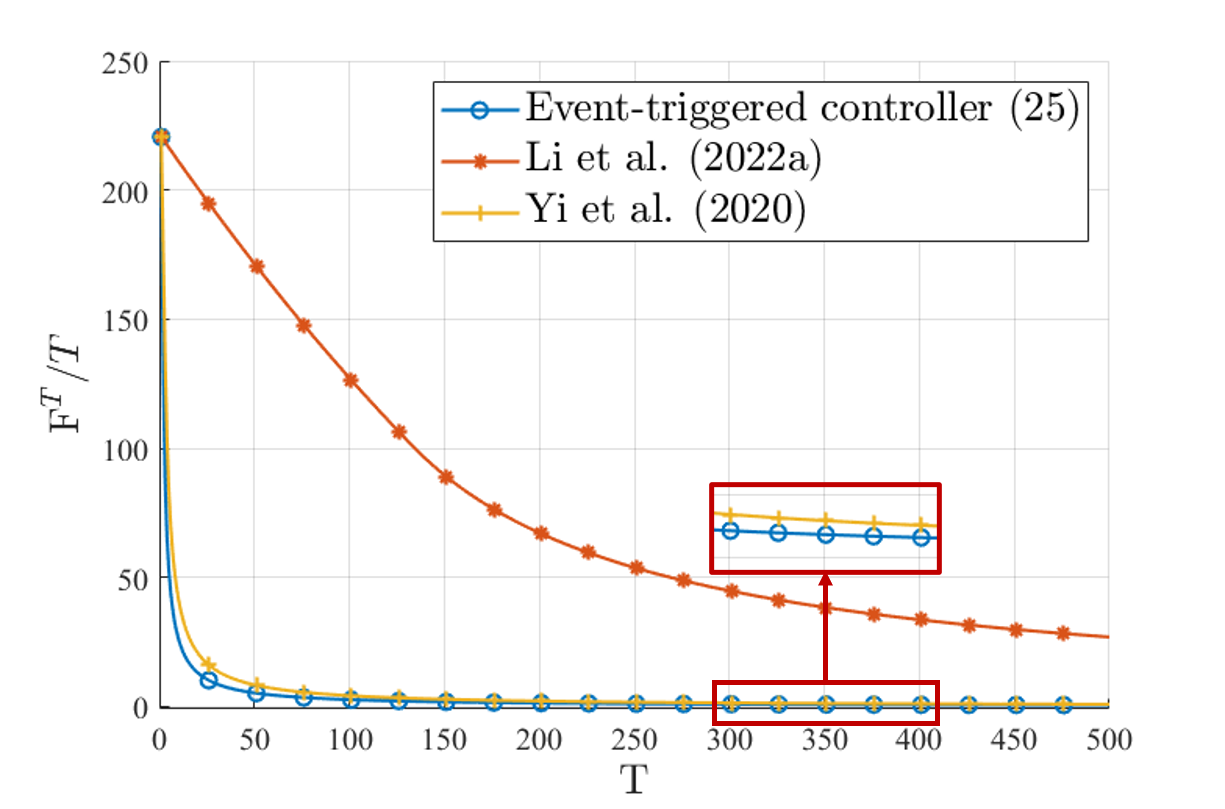} 
	\caption{ Evolution of $\mathcal{F}^T/T$.}
	\label{fig4}
\end{figure}

\begin{example}
	Since there are no distributed online algorithms to solve the
	problem of distributed online optimization with time-varying
	coupled inequality constraints for linear systems, we compare the algorithm with the distributed online algorithms \citep{yi2020, li2022}  for single integrator systems (i.e., $A_i=0$, $B_i=I$, $C_i=I$).

	We take into account a fleet of $N$ Plug-in (Hybrid) Electric Vehicles (PEVs) that need to be charged by utilizing the same electrical grid. In order to prevent undue strain on the network's lines and transformers, it becomes increasingly important to control the charging patterns of PEVs as the number of them grows. The PEV charging problem is to minimize the global cost of all PEVs in a time period, subject to some global constraints from the maximum power that can be provided by the grid, as well as local constraints for individual PEVs \citep{vuj2016}. It is assumed that the control computer of each PEV can exchange information with others in the grid and manage the charging power. Such the PEV charging problem falls within the framework of problem (\ref{question2}) and is modeled as
	\begin{align}
		\begin{split}
			\min_{\{x_i \in \mathcal{X}_i\}_{i=1}^N} &\int_{0}^{T}\sum_{i=1}^N \left(\frac{\alpha_{i,t}}{2}\|x_i\|^2+\beta_{i,t}^\top x_i\right)\,dt ,\\
			s.t. &\sum_{i=1}^{N}\left(H_{i,t}x_i+d_{i,t}\right)\le \boldsymbol0, \, t\in (0,T],\label{questione}
		\end{split}
	\end{align}
	where $\mathcal{X}_i\subseteq \mathbb{R}^{n_i}$ is the local set constraints, $x_i$ is the local decision variable of agent $i$, and $\{\alpha_{i,t}, \beta_{i,t}, H_{i,t}, d_{i,t}\}$ are time-varying and private parameters for agent $i$.

	In the simulation, an undirected and connected communication graph is randomly generated with $N=50$ nodes, as shown in Figure \ref{pic21}. Set $p_i=1$, $\mathcal{X}_i=[0,5]$ and $H_i=I$. The parameters $\alpha_{i,t}$, $\beta_{i,t}$ and $d_{i,t}$  are generated randomly from $[0.5,1]$, $[0.3,0.5]$ and $[0.5,0.7]$.

	Here,  controller (\ref{eqpi4}) in this paper with $\varepsilon=1$,  Algorithm 1 in \citet{li2022} with $\kappa=0.5$ and Algorithm 1 in \citet{yi2020} with $\kappa=0.5$ are used (the selection of $\kappa=0.5$ leads to the minimum bound according to the theoretical analyses in \citet{li2022,yi2020}). From Figure \ref{fig3} and \ref{fig4}, it can be seen that in this example event-triggered controller (\ref{eqpi4}) achieves smaller  regret and fit than the algorithm in \citet{li2022} and almost the same values as the algorithm in \citet{yi2020}.

\end{example}

\section{Conclusion}\label{res}	
	The problem of distributed online convex optimization for heterogeneous linear multi-agent systems with time-varying cost functions and time-varying coupling inequality constraints has been studied in this paper. A  continuous-time distributed controller has been proposed based on the saddle-point method and shown to have a constant regret bound and a sublinear fit bound. In order to avoid continuous communication and reduce the communication overhead, an event-triggered communication scheme has been developed which excludes Zeno behavior and also achieves a constant regret bound and a sublinear fit bound. 
	We have further extended the  controller to the case with measurement noise and demonstrated that the proposed  controller can guarantee a similar performance under certain noise interference.
	Possible future work includes   considering the impact of model uncertainty and the discrete-time version of the proposed algorithms by using exact discretization or forward Euler discretization.
\section*{Appendix}

	\subsubsection*{{A.} Proof of Lemma \ref{lemma4}}\label{ple4}

	The state estimation error can be defined as $e_x:=x-\hat{x}$. 
	Together with (\ref{eqsys}) and (\ref{ob1}), the error dynamics is obtained as 
	\begin{align}
		\dot{e}_x=(A-HC)e_x,\label{ex2}
	\end{align}
	where $H := diag(H_1,\dots,H_N)$.
		
	Calculating the time derivative of $V_1$ along the trajectories generated by (\ref{eqpi1}), one has
	\begin{align}
		\dot{V}_1
		=
		&(Cx-\tilde{y})^\top CB\Gamma \Pi_{\mathcal{Y}}[y, -\varepsilon \partial_{y}\mathcal{L}(t,y,\mu)]\notag\\
		&+ (Cx-\tilde{y})^\top C\Big(Ax-BK\hat x-(A{-}B K)\Psi\eta\Big)\notag\\
		\le 
		&\varepsilon (\tilde{y}-y)^\top \partial_{y}\mathcal{L}(t,y,\mu) + (Cx-\tilde{y})^\top CA_c(x - \Psi\eta)\notag\\
		&+ (Cx-\tilde{y})^\top BKe_x\notag\\
		\le 
		&\varepsilon\mathcal{{L}}(t,\tilde y,\mu(t))-\varepsilon\mathcal{{L}}(t,y(t),\mu(t)) -(\frac{\varepsilon l}{2}-2\varsigma_1)\|y - \tilde{y}\|^2\notag\\
		&+\frac{\|A_cC\|^2}{4\varsigma_1} \| x - \Psi\eta\|^2+ \frac{\|BK\|^2}{4\varsigma_1}\|e_x\|^2,\label{v111}
	\end{align}
	where $A_C:=A-BK$, $\varsigma_1$ is a positive parameter satisfying that $\varsigma_1\le\frac{\varepsilon l}{4}$.
	
	Similarly, the time derivative of $V_2$ along the trajectories generated by (\ref{eqpi1}) is calculated as
	\begin{align}
		\dot{V}_2
		=
		&(\mu-\tilde{\mu})^\top \Pi_{\mathbb{R}_+^q}[\mu,\varepsilon  \partial_{\mu}\mathcal{L}(t,y,\mu)]\notag\\
		\le 
		&\varepsilon\mathcal{{L}}(t,y(t),\mu(t)) -\varepsilon\mathcal{{L}}(t,y(t),\tilde{\mu}).&\label{v112}
	\end{align}
	Rewrite the system dynamics (\ref{eqpi2a}) and (\ref{ex2}) as a new system, one can get
	\begin{align}
		\begin{bmatrix} \dot{e}_x \\\dot x - \Psi\dot\eta  \end{bmatrix}
		=\underbrace{\begin{bmatrix} A-TC &\boldsymbol{0}\\BK  &A-BK \end{bmatrix}}_{A_H}
		\begin{bmatrix} e_x \\x - \Psi\eta  \end{bmatrix}.
	\end{align}
	
	Since $A-HC$ and $A-BK$ are Hurwitz, there exists a positive definite matrix $P$ for $\varsigma_2>\max\{\frac{\|A_cC\|^2}{4\varsigma_1}, \frac{\|BK\|^2}{4\varsigma_1}\}$ such that
	\begin{align}
		A_H^\top P + PA_H + \varsigma_2I < 0.\label{v113}
	\end{align}
	
	Thus, one has that
	\begin{align}
		\dot{V}\le \varepsilon \Big(\mathcal{{L}}(t,\tilde y,\mu(t))- \mathcal{{L}}(t,y(t),\tilde{\mu})\Big).\label{v114}
	\end{align}
	Integrating (\ref{v114}) from $0$ to $T$ on both sides leads to that
	\begin{align}
		&\int_{0}^{T} \Big( \mathcal{L}(t,y(t),\tilde{\mu})-\mathcal{L}(t,\tilde{y},\mu(t))\Big) \,dt\notag\\
		\le&-\frac{1}{\varepsilon}\int_{0}^{T} \dot{V}_{(\tilde{y}, \tilde{\mu})} (y(t), \mu(t)) \,dt\notag\\
		=&-\frac{1}{\varepsilon}\Big(V_{(\tilde{y}, \tilde{\mu})}(y(T), \mu(T)) - V_{(\tilde{y}, \tilde{\mu})}(y(0), \mu(0))\Big).
	\end{align}
	Because the energy function (\ref{v1}) is always nonnegative and $\mu(0)=\boldsymbol{0}$, the conclusion (\ref{l3}) can be obtained.

	\subsection*{B. Proof of Theorem \ref{th1}}\label{pth1}
	\textbf{The following is the proof of the regret bound.} 
	By choosing $\tilde{y}=y^*$ and $\tilde{\mu}=\boldsymbol{0}_{Nq}$ in Lemma \ref{lemma4}, one has
	\begin{align}
		\int_{0}^{T} \Big(\mathcal{L}(t,y,\boldsymbol{0}_{Nq}) - \mathcal{L}(t,y^*,\mu)\Big)dt\le\frac{V_{(y^*, \boldsymbol{0}_{Nq})}(x(0), \boldsymbol{0})}{\varepsilon}.\label{th1a}
	\end{align}
	According to (\ref{regert}) and (\ref{lag}), it can be obtained that
	\begin{align}
		\mathcal{R}^T{=}
		&\int_{0}^{T} \!\!\left(\mathcal{L}(t,y,\boldsymbol{0}_{Nq}) {-} \mathcal{L}(t,y^*,\mu)\right)dt
		{+}\int_{0}^{T} \!\mu^\top g(t,y^*)\, dt\notag\\
		&-\int_{0}^{T} K_\mu h(\mu)\, dt.\label{th1b}
	\end{align}	
	
	Consider the second term of the right-hand side of (\ref{th1b}). Let $\phi(\mu):=\mu^\top g(t,y^*)$ for simplicity. Then, by introducing an intermediate variable $\bar{\mu}:=\frac{1}{N}\sum_{i=1}^{N}\mu_i$ and the relationship $\sum_{i=1}^{N} g_{i}(t, y_i^*) \le \boldsymbol0$, one has that
	\begin{align}
		\phi(\mu)\le(\mu-\boldsymbol{1}_N\otimes \bar{\mu})^\top g(t,y^*).\label{phi1}
	\end{align}
	Further, one can obtain that
	\begin{align}
		\phi(\mu)^2
		&\le NK_g^2\left\|\mu-1_N\otimes \bar{\mu}\right\|^2\notag\\
		&= NK_g^2 \sum_{i=1}^{N} \Big\|\mu_i-\frac{1}{N}\sum_{j=1}^{N}\mu_j\Big\|^2\notag\\
		&\le K_g^2 \sum_{i=1}^{N} \sum_{j=1}^{N} \left\|\mu_i-\mu_j\right\|_1^2.\label{phi2}
	\end{align}
	
	Since $\mathcal{G}$ is undirected and connected,  there always exists a path connecting nodes $i_0$ and $j_0$ for any $i_0,j_0\in\mathcal{V}$, i.e.,
	\begin{align}
		h(\mu)^2=\left(\frac{1}{2}\sum_{i=1}^N\sum_{j=1}^Na_{ij}\left\|\mu_i-\mu_j\right\|_1\right)^2\ge \left\|\mu_{i_0}-\mu_{j_0}\right\|_1^2.\label{h1}
	\end{align}
	
	Then for $K_\mu \ge NK_g$, one has that
	$\phi(\mu)^2\le K_\mu^2 h(\mu)^2$,
	i.e.,
	\begin{align}
		\phi(\mu)-K_\mu h(\mu)\le 0.\label{th1c}
	\end{align}

	Substituting (\ref{th1a}) and (\ref{th1c}) into (\ref{th1b}) completes the proof.
	
	\textbf{The following is the proof of the fit bound.}
	By Lemma \ref{lemma4} with $\tilde{y}=y^*$ and $\tilde{\mu}= \boldsymbol{1}_N \otimes \gamma$, where $\gamma=col(\gamma_1,\dots,\gamma_q)\in\mathbb{R}^q$ is a parameter to be determined later, one has that
	\begin{align}
		&\int_{0}^{T} \Big(\mathcal{L}(t,y,\boldsymbol{1}_N\otimes \gamma) - \mathcal{L}(t,y^*,\mu)\Big) \,dt\notag\\
		=
		&\int_{0}^{T} \bigg(f(t,y){+}\gamma^\top\!\! \sum_{i=1}^{N}g_i(t,y_i){-}f(t,y^*){-}\mu^\top g(t,y^*) +K_\mu h(\mu) \bigg) dt \notag\\
		\le
		& \frac{V_{(\tilde{y}, \tilde{\mu})}(y(0), \boldsymbol{0})}{\varepsilon}.\label{th2a}
	\end{align}
	Invoking Assumption \ref{asp2} yields
	\begin{align}
		\int_{0}^{T} \big(f(t,y^*)-f(t,y) \big)\,dt \le 2NK_fT.\label{th2b}
	\end{align}
	
	Substituting (\ref{th1c}) and (\ref{th2b}) into (\ref{th2a}), one has that
	\begin{align}
		\int_{0}^{T} \gamma^\top \sum_{i=1}^{N}g_i(t,y_i)\,dt \le& \frac{V_{(\tilde{y}, \tilde{\mu})}(y(0), \boldsymbol{0})}{\varepsilon} + 2NK_fT.\notag
	\end{align}
	
	By choosing
	$
		\gamma_j=\left\{
		\begin{aligned}
			& 0, &F_j^T\le 0 \\
			& \frac{\varepsilon}{N}F_j^T, &F_j^T>0
		\end{aligned}
		\right.,j=[q],
	$
	it can be concluded that
	\begin{align}
		\frac{\varepsilon}{N}\|\mathcal{F}^T\|^2\le &\frac{\|y(0){-}y^*\|^2+ 2x(0)^{\!\top} \!P\! x(0) + \frac{\varepsilon^2}{N}\|\mathcal{F}^T\|^2}{2\varepsilon} + 2NK_fT.\notag
	\end{align}
	
	It can be further obtained by transposition that
	\begin{align}
		\mathcal{F}^T {\le}& \frac{\sqrt{N}\|y(0){-}y^*\| {+}\sqrt{2N x(0)^{\!\top} \!P\! x(0)}}{\varepsilon}  {+}  2N\sqrt{\frac{K_f}{\varepsilon}}\sqrt{T}.
	\end{align}	

	\subsection*{{C.} Proof of Lemma \ref{lemma6}}\label{ple6}
	Calculating the time derivative of the energy function $V_2$ together with (\ref{eqpi4}) yields
	\begin{align}
		\dot{V}_{2}
		= 
		&\!\sum_{i=1}^{N} \!(\mu_i{-}\tilde{\mu}_i\!)\!^\top \!\Pi_{\mathbb{R}_+^q}\!\Big[\mu_i,\varepsilon g_i(t,y_i) {-}2\varepsilon K\!_\mu\! \sum_{j=1}^N \!a_{i\!j}\sign{\hat\mu_i{-}\hat\mu_j}\Big] \notag\\
		\le
		& \underbrace{\varepsilon\! \sum_{i=1}^{N} (\mu_i{-}\tilde{\mu}_i\!)\!^\top\! \mathcal{L}^{\mu_i}(t,\!y_i,\!\mu_i)}_{S_{1}}
		 \underbrace{ -2\varepsilon K_\mu\!\! \sum_{i=1}^{N} \!(\mu_i{-}\tilde{\mu}_i)\!^\top\!\! \sum_{j=1}^N \!\!a_{ij}\sign{\hat\mu_i-\hat\mu_j} }_{S_{2}} \notag\\
		&+ \underbrace{ \varepsilon K_\mu \!\sum_{i=1}^{N} \!(\mu_i{-}\tilde{\mu}_i)\!^\top  \sum_{j=1}^N a_{ij}\sign{\mu_i-\mu_j} }_{S_{3}}.	\label{vb1}
	\end{align}
	
	Since the Lagrangian (\ref{la}) is convex with respect to $y_i$ and concave with respect to $\mu_i$, one has that
	\begin{align}
		S_{1}+ \dot{V}_{1} + \dot{V}_{3} \le \varepsilon \Big(\mathcal{L}(t,\tilde{y},\mu) - \mathcal{L}(t,y,\tilde{\mu})\Big).\label{vb2}
	\end{align}
	
	Since $\tilde{\mu}= \boldsymbol1_N \otimes \gamma, \forall \gamma\in \mathbb{R}_+^q$ and graph $\mathcal{G}$ is undirected, it follows that
	\begin{align}
		S_{2}=
		&-\varepsilon K_\mu \sum_{i=1}^{N}  \sum_{j=1}^N a_{ij}(\mu_i-\mu_j)^\top \sign{\hat\mu_i-\hat\mu_j}\notag\\
		\le
		& \!-\!\varepsilon K_\mu \!\sum_{i=1}^{N}\!\sum_{j=1}^N a_{ij} \|\hat\mu_i{-}\hat\mu_j\|_1 \!+\! \varepsilon K_\mu\!\sqrt{q} \sum_{i=1}^{N}  \!\sum_{j=1}^N \|e_i{-}e_j\|,	\label{vb3}
	\end{align}
	where the last inequality holds since the relationship  $\|v\|_1\le\sqrt{q}\|v\|, \forall v\in\mathbb{R}^q$.
	
	Similarly, one has that
	\begin{align}
		S_{3}
		=
		&\frac{\varepsilon K_\mu}{2} \sum_{i=1}^{N} \sum_{j=1}^N a_{ij}\|\mu_i-\mu_j\|_1\notag\\
		\le	
		&\frac{\varepsilon K_\mu}{2} \sum_{i=1}^{N} \sum_{j=1}^N a_{ij}\|\hat\mu_i{-}\hat\mu_j\|_1 + \frac{\varepsilon K_\mu\sqrt{q}}{2} \sum_{i=1}^{N} \sum_{j=1}^N \|e_i{-}e_j\|.\label{vb4}
	\end{align}
	
	For the second item in (\ref{vb3}) and (\ref{vb4}), one can calculate that
	
	\begin{align}
		\sum_{i=1}^{N}  \sum_{j=1}^N  \|e_i-e_j\|
		\le& \sum_{i=1}^{N}  \sum_{j=1}^N  \|e_i\| + \sum_{i=1}^{N}  \sum_{j=1}^N  \|e_j\|\notag\\
		\le &\frac{1}{3\sqrt{q}}\sum_{i=1}^{N} \sum_{j=1}^N a_{ij} \|\hat\mu_i-\hat\mu_j\|_1 + \frac{2\sigma e^{-\iota t}}{3K_\mu\sqrt{q}} ,	
	\end{align}
	where the last inequality holds due to the trigger condition (\ref{tau2}).
	
	Summarizing the above-discussed analysis,	$\dot{V}_{(\tilde{y}, \tilde{\mu})}$ in (\ref{vb1}) is calculated as
	\begin{align}
		\dot{V}_{(\tilde{y}, \tilde{\mu})}
		\le
		& \varepsilon \Big(\mathcal{L}(t,\tilde{y},\mu) - \mathcal{L}(t,y,\tilde{\mu})\Big) + \varepsilon\sigma e^{-\iota t}.\notag
	\end{align}
	
	By integrating it from $0$ to $T$ on both sides and omitting negative terms, it can be obtained that
	\begin{align}
		\int_{0}^{T} \Big( \mathcal{L}(t,y,\tilde{\mu}){-}\mathcal{L}(t,\tilde{y},\mu)\Big) dt {\le} \frac{V_{(\tilde{y}, \tilde{\mu})}(y(0), \boldsymbol{0})}{\varepsilon} {+} \frac{\sigma}{\iota}.
	\end{align}
	
	\subsection*{{D.} Proof of Theorem \ref{th3}}	\label{pth3}
	The next proof is that the controller (\ref{eqpi4}) has no Zeno behavior.	
	In the trigger interval $[t_i^l,t_i^{l+1})$ for agent $i$, combining the definition of $e_i$ with (\ref{eqpi4b}), one can write the upper right-hand Dini derivative as
	\begin{align}
		D^+ e_i(t){=} - \!\Pi_{\mathbb{R}_+^q}[\mu_i, \varepsilon g_i(t,y_i) {-}2\varepsilon K_\mu \sum_{j=1}^N a_{ij}\sign{\hat\mu_i{-}\hat\mu_j}].\label{zeno1}
	\end{align}
	It is obvious that $e_i(t_i^l)=0$. Then, for $t\in (t_i^l,t_i^{l+1})$, the solution of (\ref{zeno1}) is
	\begin{align}
		e_i(t){=} \!\int_{t_i^l}^{t} \!- \!\Pi_{\mathbb{R}_+^q}[\mu_i, \varepsilon g_i(t,y_i) {-}2\varepsilon K_\mu \sum_{j=1}^N a_{ij}\sign{\hat\mu_i{-}\hat\mu_j}] \,d\tau. \label{zeno2}
	\end{align}
	From Assumption \ref{asp2} and the inequality $\|\Pi_{\mathcal{S}}[x,v]\|\le\|v\|$ (cf. Remark 2.1 in \citet{zhang1995}), it can be obtained that the norm of the integral term in (\ref{zeno2}) is bounded, and let the upper bound of its norm be $\delta$.  It then follows from (\ref{zeno2}) that
	\begin{align}
		\|e_i(t)\| \le \delta(t-t_i^l).
	\end{align} 
	Hence, condition (\ref{tau2}) will definitely not be triggered before the following  condition holds:
	\begin{align}
		\delta(t-t_i^l)=\frac{1}{6N\sqrt{q}}\sum_{j=1}^N a_{ij} \|\hat\mu_i-\hat\mu_j\|_1 {+} \frac{\sigma e^{-\iota t}}{3N^2K_\mu\sqrt{q}} .\label{zeno3}
	\end{align}
	
	It is easy to obtain that the right-hand side of (\ref{zeno3}) is positive for any $t$,  which further implies that $t-t_i^l>0$. Hence, the value $t_i^{l+1}-t_i^l$ is strictly positive for any $t$. Furthermore, since the right-hand side of (\ref{zeno3}) approaches zero only when $t\to\infty$, one has that $t_i^l\to\infty$ with $l\to\infty$. Therefore, no Zeno behavior is exhibited.

	\subsection*{{E.} Proof of Lemma \ref{lemma5}}\label{ple5}
	Let us	denote by $ \partial_{y_i}\mathcal{\tilde{L}}(t,y,\mu)$ a subgradient of $\mathcal{\tilde{L}}$ with respect to $y_i$, i.e.,
	\begin{align}
		 \partial_{y_i}\mathcal{\tilde{L}}(t,y,\mu)\in \partial f_i(t,y_i)+\mu_i^\top \partial g_i(t,y_i).\label{lany}
	\end{align}
	It can be seen that $ \partial_{y_i}\mathcal{\tilde{L}}(t,y,\mu)= \partial_{y_i} \mathcal{{L}}(t,y,\mu)$.
	
	Similarly, denote by $ \partial_{\mu_i}\mathcal{\tilde{L}}(t,y,\mu)$ a subgradient of $\mathcal{\tilde{L}}$ with respect to $\mu_i$, i.e.,
	\begin{align}
		& \partial_{\mu_i}\mathcal{\tilde{L}}(t,y,\mu){\in}  -\! \frac{K_\mu}{2}\!\sum_{j=1}^N \! a_{ij}\Big[\sign{\mu_i{-}\mu_j {+} \epsilon_{ij}(t)\|\mu_i(t){-}\mu_j(t)\|}\notag\\
		& + \sign{\mu_i{-}\mu_j {-} \epsilon_{ij}(t)\|\mu_i(t){-}\mu_j(t)\|}\Big] + g_i(t,y_i).\label{lanmu}
	\end{align}
	
	Because the probability density function of $\epsilon_{ij}(t)$ is symmetric, one has
	\begin{align}
		\mathbb{E}\Big[\sign{\mu_i(t)-\mu_j(t)+\epsilon_{ij}(t)\|\mu_i(t)-\mu_j(t)\|}\Big|\mu(t)\Big]\notag\\
		=\mathbb{E}\Big[\sign{\mu_i(t)-\mu_j(t)-\epsilon_{ij}(t)\|\mu_i(t)-\mu_j(t)\|}\Big|\mu(t)\Big],\label{condi1}
	\end{align}
	and it can be further obtained that
	\begin{align}
		\mathbb{E}\Big[\dot{\mu}_i(t)\Big|\mu(t)\Big]=\mathbb{E}\Big[ \partial_{\mu_i}\mathcal{\tilde{L}}(t,y,\mu(t))\Big|\mu(t)\Big].
	\end{align}

	Calculating the time derivative of the energy function $V_2$ together with (\ref{eqpi7}) yields	
	\begin{align}
		\dot{V}_{2}
		{=} 
		&(\mu(t)-\tilde{\mu})^\top \dot{\mu}(t)\notag\\
		=
		&\sum_{i=1}^{N} (\mu_i(t)\!-\!\tilde{\mu}_i)\!^{\top} \Pi_{\mathbb{R}_+^q}\Big[\mu_i(t),\varepsilon g_i(t,y_i(t))\notag\\
		&-\varepsilon K_\mu \!\sum_{j=1}^N a_{ij} \sign{\mu_i(t){-}\mu_j(t){+}\epsilon_{ij}(t)\|\mu_i(t){-}\mu_j(t)\|}\Big] \notag\\
		\le
		& \sum_{i=1}^{N} (\mu_i(t){-}\tilde{\mu}_i)^{\!\top}\! \Big(\varepsilon g_i(t,y_i(t)){-}\varepsilon K_{\!\mu}\!\! \sum_{j=1}^N a_{ij} \sign{\mu_i(t)\notag\\
			&{-}\mu_j(t) + \epsilon_{ij}(t)\|\mu_i(t){-}\mu_j(t)\|}\Big),\label{v42}
	\end{align}
	where the inequality  holds because of  (\ref{eqle2}).
	
	Taking conditional expectations of (\ref{v42}) at $\mu(t)$ and taking advantage of (\ref{condi1}), one has
	\begin{align}
		\mathbb{E}\big[\dot{V}_{2}\big|\mu(t)\big]
		\le	
		&\varepsilon\sum_{i=1}^{N} \mathbb{E}\Big[(\mu_i(t)-\tilde{\mu}_i)^\top 
		 \partial_{\mu_i}\mathcal{\tilde{L}}(t,y(t),\mu(t))\Big|\mu(t)\Big].\label{v43}	
	\end{align}
	
	Taking expectations on both sides of (\ref{v43}) and making use of the concavity of Lagrangian (\ref{lan}) with respect to $\mu_i$, it can be obtained that
	\begin{align}
		\mathbb{E}\big[\dot{V}_{2}\big]
		\le &\varepsilon\mathbb{E}\Big[\sum_{i=1}^{N}(\mu_i(t)-\tilde{\mu}_i)^\top \partial_{\mu_i}\mathcal{\tilde{L}}(t,y(t),\mu(t))\Big]\notag\\
		\le
		&\varepsilon\mathbb{E}\big[\mathcal{\tilde{L}}(t,y(t),\mu(t))-\mathcal{\tilde{L}}(t,y(t),\tilde{\mu})\big].\label{v44}
	\end{align}
	

	In this manner, one has
	\begin{align}
		\mathbb{E}\big[\dot{V}\big]
		=
		&\mathbb{E}\big[\dot{V}_{1}\big]+\mathbb{E}\big[\dot{V}_{2}\big] + \mathbb{E}\big[\dot{V}_{3}\big]\notag\\
		\le
		&\varepsilon \mathbb{E}\big[\mathcal{\tilde{L}}(t,\tilde y,\mu(t))-\mathcal{\tilde{L}}(t,y(t),\tilde{\mu})\big].\label{v45}
	\end{align}

	Integrating (\ref{v45}) from $0$ to $T$ on both sides leads to that
	\begin{align}
		&\int_{0}^{T} \mathbb{E}\big[ \mathcal{\tilde{L}}(t,y(t),\tilde{\mu}) - \mathcal{\tilde{L}}(t,\tilde y,\mu(t)) \big] \,dt
		\le&\frac{1}{\varepsilon}		
		V_{(\tilde{y}, \tilde{\mu})}(y(0), \mu(0))
		.
	\end{align}
	Because the energy function (\ref{v1}) is always nonnegative and $\mu(0)=\boldsymbol{0}$, the conclusion (\ref{l3}) can be obtained.

	\subsection*{F. Proof of Theorem \ref{th7}}	\label{pth7}

	By choosing $\tilde{y}=y^*$ and $\tilde{\mu}=\boldsymbol{0}_{Nq}$ in Lemma \ref{lemma5}, one has
	\begin{align}
		\int_{0}^{T} \mathbb{E}\big[ \mathcal{\tilde{L}}(t,y,\boldsymbol{0}) - \mathcal{\tilde{L}}(t,y^*,\mu(t)) \big] \,dt
		\le
		\frac{V_{(y^*, \boldsymbol{0})}(y(0), \mu(0))
		}{\varepsilon}.\label{th4a}
	\end{align}
	According to (\ref{regert}) and (\ref{lag}), it can be obtained that
	\begin{align}
		&f(t,y(t))-f(t,y^*)\notag\\
		=&\mathcal{\tilde{L}}(t,y(t),\boldsymbol{0}) - \mathcal{\tilde{L}}(t,y^*\!,\mu(t)) + \mu(t)^{\!\top} \!g(t,y^*) \notag\\
		&+K_\mu\tilde{h}(0)-K_\mu\tilde{h}(\mu(t)).	\label{v46}
	\end{align}
	
	Similar to the analysis of (\ref{phi1}) and (\ref{phi2}), one can obtain that
	\begin{align}
		\phi(\mu)
		=\mu(t)^\top g(t,y^*)
		\le K_g \sum_{i=1}^{N} \sum_{j=1}^{N} \left\|\mu_i-\mu_j\right\|.\label{phi3}
	\end{align}
	
	In the same way of (\ref{h1}), one has
	\begin{align}
		&\tilde{h}(\mu(t))- \tilde{h}(0)
		\ge
		&\sum_{i=1}^N\sum_{j=1}^Na_{ij} \left(\frac{1}{2}-\|\epsilon_{ij}(t)\|_1\right) \left\|\mu_i(t)-\mu_j(t)\right\|.\label{h2}
	\end{align}
	
	Taking expectations on both sides of (\ref{h2}), one has
	\begin{align}
		&\mathbb{E}\big[\tilde{h}(\mu(t))- \tilde{h}(0)\big]\notag\\
		\ge
		&\left(\frac{1}{2}- \mathbb{E}\big[\|\epsilon_{ij}(t)\|_1\big] \right) \sum_{i=1}^N\sum_{j=1}^Na_{ij} \mathbb{E}\big[\|\mu_{i_0}-\mu_{j_0}\|\big] \notag\\
		\ge
		&\left(1- 2\mathbb{E}\big[\|\epsilon_{ij}(t)\|_1\big] \right) \mathbb{E}\big[\|\mu_{i_0}-\mu_{j_0}\|\big]\notag\\
		=
		&\frac{1}{N^2}\left(1- 2\mathbb{E}\big[\|\epsilon_{ij}(t)\|_1\big] \right) \sum_{i=1}^N\sum_{j=1}^N \mathbb{E}\big[\|\mu_i(t){-}\mu_j(t)\|\big]. \label{h3}
	\end{align}
	where the first inequality holds since the independence between $\epsilon_{ij}(t)$ and $\mu(t)$, and the second inequality holds since the similar way to (\ref{h1}).
	
	Taking expectations on both sides of (\ref{v46}) and substituting (\ref{phi3}), (\ref{h3}) into (\ref{v46}), one has
	\begin{align}
		&\mathbb{E}[f(t,y(t))]-f(t,y^*)\notag\\
		\le
		&\mathbb{E}\left[\mathcal{\tilde{L}}(t,y(t),\boldsymbol{0}) - \mathcal{\tilde{L}}(t,y^*,\mu(t))\right]\notag\\
		&+\Big(K_g{-}\frac{K_\mu}{N^2}(1{-}2\mathbb{E}\left[\|\epsilon_{ij}(t)\|_1\right])\Big)\sum_{i=1}^{N} \sum_{j=1}^{N} \mathbb{E}\big[\|\mu_i(t){-}\mu_j(t)\|\big]\notag\\
		\le
		&\mathbb{E}\left[\mathcal{\tilde{L}}(t,y,\boldsymbol{0}) - \mathcal{\tilde{L}}(t,y^*,\mu(t))\right],\label{v47}
	\end{align}
	where the last inequality holds because of the assumption stated in Theorem \ref{th7} that $\mathbb{E}\left[\|\epsilon_{ij}(t)\|_1\right]\le \frac{1}{2}-\frac{N^2K_g}{2K_\mu}$.

	Integrating (\ref{v47}) from $0$ to $T$ on both sides leads to the regret bound. The proof of fit bound follows similarly as in Theorem \ref{th1}.

\bibliographystyle{apalike2}
\bibliography{bibfile}           

\end{document}